\documentclass[11pt]{amsart}
\usepackage{fullpage}
\usepackage{amsfonts}
\usepackage{amsmath}
\usepackage{amssymb}
\usepackage{amsthm}
\usepackage{verbatim}
\usepackage{enumitem}

\newtheorem{theorem}{Theorem}

\newtheorem{cor}[theorem]{Corollary}
\newtheorem*{question}{Question}
\newtheorem*{CT}{Stopping Criterion for SOS Relaxations}

\newtheorem{lemma}[theorem]{Lemma}
\newtheorem{prop}[theorem]{Proposition}

\newcommand{\codim}{\operatorname{codim}}

\newcommand{\rn}{\mathbb{R}^{n}}
\newcommand{\RR}{\mathbb{R}}
\newcommand{\PP}{\mathbb{P}}
\newcommand{\CC}{\mathbb{C}}
\newcommand{\CP}{\mathbb{CP}}
\newcommand{\rxt}{\mathbb{R}[x_1,x_2,x_3]}

\newcommand{\p}{P_{n,2d}}
\newcommand{\h}{\mathbb{R}[x]_{n,2d}}
\newcommand{\rmm}{\mathbb{R}^{m}}
\newcommand{\Ind}{\operatorname{Ind}}
\newcommand{\Dep}{\operatorname{Dep}}

\newcommand{\co}{\mathbb{C}}
\newcommand{\hd}{\mathbb{R}[x]_{n,d}}

\newcommand{\V}{\mathcal{V}}
\newcommand{\sq}{\Sigma_{n,2d}}

\newcommand{\st}{\hspace{2mm} \mid \hspace{2mm}}

\newcommand{\ql}{Q_{\ell}}
\newcommand{\wl}{W_{\ell}}
\newcommand{\bl}{B_{\ell}}
\newcommand{\rx}{\mathbb{R}[x]}
\newcommand{\cx}{\mathbb{C}[x]}

\newcommand{\rank}{\operatorname{rank}}
\numberwithin{theorem}{section}
\numberwithin{equation}{section}
\begin{document}

\title{Positive Gorenstein Ideals}
\author{Grigoriy Blekherman}
\begin{abstract}
We introduce positive Gorenstein ideals. These are Gorenstein ideals in the graded ring $\RR[x]$ with socle in degree $2d$, which when viewed as a linear functional on $\rx_{2d}$ is nonnegative on squares. Equivalently, positive Gorenstein ideals are apolar ideals of forms whose differential operator is nonnegative on squares. Positive Gorenstein ideals arise naturally in the context of nonnegative polynomials and sums of squares, and they provide a powerful framework for studying concrete aspects of sums of squares representations.  We present  applications of positive Gorenstein ideals in 
real algebraic geometry, analysis and optimization. In particular, we present a simple proof of Hilbert's nearly forgotten result on representations of ternary nonnegative forms as sums of squares of rational functions. Drawing on our previous work in \cite{GB2012}, our main tools are Cayley-Bacharach duality and elementary convex geometry.
\end{abstract}

\maketitle

\section{Introduction}

A real polynomial in $n$ variables is called nonnegative if it is greater than or equal to $0$ on all points in $\rn$. The relationship between nonnegative polynomials and sums of squares (of polynomials or more general objects) is a fundamental question in real algebraic geometry. Algorithmic approaches to this question have been quite useful in polynomial optimization \cite{JiawangMarkus},\cite{Lasserre},\cite{LasserreBook},\cite{pablo}.




Any nonnegative polynomial can be made homogeneous by adding an extra variable and it will remain nonnegative. The same holds for sums of squares. We will therefore work with homogeneous polynomials (forms). 

Let $\hd$ be the vector space of real forms in $n$ variables of degree $d$. Nonnegative forms and sums of squares both form full dimensional closed convex cones in $\h$, which we call $\p$ and $\sq$ respectively:

$$\p=\left\{p \in \h \hspace{2mm} \mid \hspace{2mm} p(x)\geq 0 \hspace{2mm} \text{for all} \hspace{2mm} x \in \rn\right\},$$

\noindent and
$$\sq=\left\{p \in \h \hspace{2mm} \big| \hspace{2mm} p(x)=\sum q_i^2  \hspace{2mm} \text{for some} \hspace{2mm} q_i \in \hd \right\}.$$

In 1888 Hilbert showed that nonnegative polynomials are the same as sums of squares only in the following three cases: $n=2$, $2d=2$ and $n=3$, $2d=4$. In all other cases $P_{n,2d}$ is strictly large than $\sq$ \cite{Hilbert1}.

The defining linear inequalities of a cone are described by its dual cone. Let $\p^*$ and $\sq^*$ be the dual cones of $\p$ and $\sq$:
$$\p^*=\left\{\ell \in \h^* \hspace{3mm} | \hspace{3mm} \ell(p) \geq 0 \hspace{5mm} \text{for all} \hspace{5mm} p \in \p \right\},$$
\noindent and
$$\sq^*=\left\{\ell \in \h^* \hspace{3mm} | \hspace{3mm} \ell(p) \geq 0 \hspace{5mm} \text{for all} \hspace{5mm} p \in \sq \right\}.$$

\noindent The dual cone $\p^*$ is easy to describe: its extreme rays are point evaluations, i.e. the linear functionals $\ell_v \in \h^*$ given by $$\ell_v(f)=f(v) \hspace{7mm} \text{for} \hspace{5mm} f\in \h \hspace{5mm} \text{and} \hspace{5mm} v \in \rn.$$

Linear functionals $\ell_v$ are also extreme rays of $\sq^*$, but when nonnegative polynomials are not the same as sums of squares, there are linear functionals that are nonnegative on squares and do not come from point evaluations. In order to understand properties of such linear functionals we look at the Gorenstein ideals that they generate. To a linear functional $\ell: \h \rightarrow \RR$ we can associate the Gorenstein ideal $I(\ell)$: 
$$I(\ell)=\{ p \in \rx\,\, \mid \,\,  \deg (p)> 2d \hspace{5mm} \text{or} \hspace{3mm} \ell(pq)=0 \hspace{3mm} \text{for all} \hspace{3mm} q \in \rx_{2d-\deg(p)}\}.$$
By a slight abuse of terminology we will call $\ell$ the socle of $I(\ell)$.

We can identify $\h$ with its dual space $\h^*$ by sending $x_i$ to the differential operator $\frac{\partial}{\partial x_i}$ and replacing multiplication with composition of differential operators. For $f \in \h$ let $\partial f \in \h^*$ denote the corresponding differential operator. Then the Gorenstein ideal $I(\partial f)$ corresponds precisely to the \textit{apolar ideal} of $f$. We remark that with identification of $\h$ and $\h^*$ the dual cone $\p^*$ is the conic hull of the real Veronese variety of degree $2d$, and thus the dual cone of $\p$ is the conical hull of the Veronese Orbitope \cite{Rez3},\cite{Raman}. In the following we prefer to keep the language of linear functionals and Gorenstein ideals.


We call a positive Gorenstein ideal $I$ with socle $\ell$ of degree $2d$ \textit{maximal}, if $\ell$ is an extreme ray of $\Sigma_{n,2d}$. We showed in \cite{GB2012} that $I(\ell)$ is a maximal positive Gorenstein ideal if and only if $I_d(\ell)$ is maximal (by inclusion) over all Gorenstein ideals with socle of degree $2d$. We provided a classification of extreme rays of $\sq^*$ for the two smallest cases where there exist nonnegative polynomials that are not sums of squares: $n=3$, $2d=6$ and $n=4$, $2d=4$. We summarize some of results of \cite{GB2012} in the language of positive Gorenstein ideals:

\begin{theorem}\label{THM Prev Main}
Let $I$ be a maximal positive Gorenstein ideal with socle $\ell \in \rx_{3,6}^*$ (resp. $\rx_{4,4}^*$). Then $I$ is a complete intersection of three cubics (resp. four quadrics).

Let $S$ be the linear span of the three cubics (resp. four quadrics) in $I$. Then any two cubics (resp. three quadrics) in $S$ intersecting transversely contain at most one conjugate pair of complex zeroes, with the rest of zeroes being real.
\end{theorem}

This characterization was later used in \cite{HCases} to study the \textit{algebraic boundaries} of the cones $\Sigma_{3,6}$ and $\Sigma_{4,4}$ via a surprising connection with K3 surfaces.

We now undertake the study of positive Gorenstein ideals and present applications in real algebraic geometry, analysis and optimization. Our main tool is Theorem \ref{THM Master} which uses Cayley-Bacharach duality of \cite{EGH} and allows us to give a unified presentation.  

\section{Results.}\label{SEC Results}

Let $I$ be a maximal positive Gorenstein ideal with socle of $\ell$ of degree $2d$. We first investigate the possible dimensions of $I_d$. As can be seen from Theorem \ref{THM Prev Main}, we expect maximal positive Gorenstein ideals to possess significant structure, and the choice of possible dimensions for $I_d$ should be limited.


\begin{theorem}\label{THM Main Rank}
Let $I$ be a maximal positive Gorenstein ideal in $\RR[x]$ with socle $\ell$ of degree $2d$, which is not a point evaluation. Then the forms in $I_d$ have no common zeroes, real or complex, and $I_{d}$ generates $I_{2d}$. Additionally we have,
$$\codim I_d \geq 3d-2 \hspace{3mm} \text {if} \hspace{3mm} d \geq 3 \hspace{5mm} \text{or} \hspace{3mm}\codim I_{d}\geq 6 \hspace{3mm}\text{if} \hspace{3mm} d=2.$$ 

The bounds are tight for a complete intersection of a cubic and two forms of degree $d$ in $\rxt$ with $d \geq 3$, and a complete intersection of four quadrics in $\RR[x_1,x_2,x_3,x_4]$ for $d=2$.
\end{theorem}

To a linear functional $\ell: \h \rightarrow \RR$ we associate a quadratic form $\ql$ on $\hd$ by setting $\ql(p)=\ell(p^2)$ for $p \in\hd$. It is easy to see that $I(\ell)_d$ is equal to the kernel of $\ql$. Restated in convex geometry terms, Theorem \ref{THM Main Rank} immediately leads to the following Corollary.

\begin{cor}\label{COR Extreme Rays}
Let $\ell \in \h^*$ span an extreme ray of $\sq^*$. If $\rank \ql \leq 3d-3$ for $d \geq 3$ or $\rank \ql \leq 6$ for $d=2$, then $\ell$ is a point evaluation. Furthermore these bounds are tight and there exist extreme rays of $\sq$ of rank $3d-2$ for $d\geq 3$, $n \geq 3$ and $6$ for $d=2$, $n \geq 4$ that do not come from point evaluations.
\end{cor}

Below we present some direct consequences of Theorem \ref{THM Main Rank} in analysis, algebraic geometry and optimization. We hope that the different interpretations highlight the interdisciplinary nature of convex algebraic geometry.

The \textit{truncated moment problem} in real analysis asks for a classification of linear functionals $\ell \in \h^*$ that come from integration with respect to a Borel measure: $\ell(f)=\int_{\RR^n}f\,d\mu$ \cite{LasserreBook}. For $\ell \in \h^*$ the matrix $M_{\ell}$ of the associated quadratic form $\ql$ with respect to the monomial basis is known as the \textit{moment matrix} of $\ell$. If $\ell$ comes from integration with respect to a measure then $M_{\ell}$ must be positive semidefinite. We show that if rank of $M_{\ell}$ is sufficiently small, then $\ell$ indeed comes from a measure, and furthermore $\ell$ can be written as a sum of precisely $\rank M_{\ell}$ point evaluations.

\begin{theorem}\label{THM Moments}
Let $\ell: \rx_{n,2d} \rightarrow \RR$ be a linear functional and suppose that the the moment matrix $M_{\ell}$ is positive semidefinite and $\rank M_{\ell} \leq 3d-3$ with $d \geq 3$ or $\rank M_{\ell} \leq 6$ for $d=2$. Then $\ell$ comes from integration with respect to a measure and it can be written as a conical combination of $\rank M_{\ell}$ point evaluations. Furthermore this bound is tight, and there exist linear functionals $\ell:  \rx_{n,2d} \rightarrow \RR$ of rank $3d-2$ for $d \geq 3$, $n \geq 3$ and rank $6$ for $d=2$, $n \geq 4$, such that $M_{\ell}$ is positive semidefinite, but $\ell$ does not come from integration with respect to a measure.
\end{theorem}

The \textit{symmetric tensor decomposition problem} (also known as the \textit{Waring problem}) for a given $f \in \hd$ (or $f \in \cx_{n,d}$) asks for the minimum number of linear forms $\ell_i$ such that $f=\sum c_i \ell_i^{d}$. The minimal number of linear forms is known as the \textit{Waring rank} of $f$ \cite{LandsbergTeitler}. By using the identification of $f \in \h$ with $\partial f \in \h^*$ we can associate to $f$ a quadratic form $Q_f: \hd \rightarrow \RR$ by setting $Q_f(p)=\partial f(p^2)$. The matrix of the quadratic form $Q_f$ is known as the \textit{middle catalecticant matrix} of $f$. It is well known that the Waring rank of $f$ is at least the rank of $Q_f$ \cite{IarrKanev}. Theorem \ref{THM Main Rank} implies that if $Q_f$ is positive semidefinite and its rank is sufficiently small, then it can be decomposed as linear combination of $2d$-th powers of real linear forms, with strictly positive coefficients, and the Waring rank of $f$ is precisely $\rank Q_f$. We remark that the \textit{real Waring problem} where we require that all the forms used in the decomposition are real tends to be more complicated than the Waring problem over $\CC$  \cite{Real Rank Boij},\cite{Real Rank Ottaviani}.
\begin{theorem}\label{THM Tensors}
The $f \in \rx_{n,2d}$ be a form such that the middle catalecticant matrix $Q_f$ of $f$ is positive semidefinite and $\rank Q_f \leq 3d-3$ if $d \geq 3$ or $\rank Q_f \leq 6$ if $d=2$. Then $f$ can be decomposed as a sum of $2d$-th powers of linear forms with positive coefficients, and the Waring rank of $f$ is equal to $\rank Q_f$:
$$f=\sum_{i=1}^{\rank Q_f} c_i\ell_i^{2d} \hspace{5mm} \text{with} \hspace{5mm} \ell_i \in \rx_{n,1} \hspace{5mm} \text{and} \hspace{5mm}c_i\geq 0.$$
Furthermore this bound is tight, and there exist forms $f \in \h$ with positive semidefinite form $Q_f$ of rank $3d-2$ whose real Waring rank is strictly greater than $\rank Q_f$.
\end{theorem}
We note that if positive semidefiteness assumption of $Q_f$ is dropped, then it is possible, using for instance lower bounds on the Waring rank from \cite{LandsbergTeitler}, to construct forms $f$ with small rank of $Q_f$, such that the Waring rank is strictly greater than the rank of $Q_f$, even if we allow $2d$-th powers of complex linear forms in the decomposition.

Theorem \ref{THM Main Rank} also leads to an interesting stopping criterion for sum of squares relaxations. Sum of squares methods lead to a hierarchy of relaxations indexed by degree. Under a variety of conditions the hierarchy is guaranteed to approach the optimal solution \cite{Lasserre},\cite{LasserreBook},\cite{pablo}. One of the important questions in this area is as follows: when can we guarantee that we obtained the actual optimum, and thus stop computing relaxations of higher degree?
Primal-dual methods for semidefinite programming, when applied to sums of squares relaxations, lead to a solution, along with a certifying dual linear functional on $\h$. If we can conclude that the certifying functional comes from a measure, then the sum of squares relaxation is exact and we obtained the optimal solution. 
From Theorem \ref{THM Moments} we see the following stopping criterion:

\begin{CT} Suppose that a sum of squares relaxation truncated in degree $2d$ with $d \geq 3$ returns an optimal linear functional with moment matrix of rank at most $3d-3$. Then the relaxation is exact. \end{CT}

The codimension bound in Theorem \ref{THM Main Rank} comes from the following Theorem, which we hope is interesting in itself.

\begin{theorem}\label{THM mainrank2}
Let $S$ be a subspace of $\cx_{n,d}$ with $n,d \geq 3$ such that $\V(S)=\emptyset$. Let $I=I(S)$ be the ideal generated by $S$ and suppose that $\codim I_{2d} \geq 1$, so that $S$ does not generate all forms of degree $2d$. Then $$\codim I_d \geq 3d-2.$$ 
The bound is tight when $S$ is the degree $d$ part of the complete intersection of a form of degree $3$ and two forms of degree $d$ in $\CP^2$.

Let $S$ be a subspace of $\cx_{n,2}$ with $n \geq 4$ such that $\V(S)=\emptyset$. Let $I=I(S)$ be the ideal generated by $S$ and suppose that $\codim I_{4} \geq 1$. Then $$\codim I_2 \geq 6.$$ 
The bound is tight when $S$ is the complete intersection of $4$ quadrics in $\CP^3$.

\end{theorem} 

We remark that the tight examples in Theorem \ref{THM mainrank2} come from three or four variate constructions, i.e. $I_1$ has codimension $3$ for $d \geq 3$ or $4$ for $d=2$. If we want to instead bound the size of $I_d$ as the function of the codimension of $I_1$, then better bounds are possible. For instance it was shown in \cite{GorDeg4} that for the case of socle of degree $4$ we have $$\lim_{(\codim I_1) \rightarrow \infty} \min \frac{\codim I_2}{(\codim I_1)^{2/3}}=6^{2/3}$$ 
for all Gorenstein ideals with socle of degree $4$. We are not aware of any results for socles of higher degree, but these would be very interesting.
\subsection{Positive forms and Positive Gorenstein Ideals.} In 1893 Hilbert showed the following \cite{Hilbert2}:

\begin{theorem}[Hilbert]\label{THM Hilbert}
Let $p$ be a nonnegative form of degree $2d$ in $3$ variables. Then there exists a nonnegative form $q$ of degree $2d-4$ such that $pq$ is a sum of squares.
\end{theorem}

By reducing degrees of multipliers $q$ Hilbert concluded that for every $p \in P_{3,2d}$ there exists a sum of squares multiplier $q$ such that $pq$ is a sum of squares, where the degree of $q$ is at most $d(d/2-1)$ if $d$ is even, and at most $(d-1)^2/2$ if $d$ is odd. This allowed Hilbert to conclude that any $p \in P_{3,2d}$ is a sum of squares of rational functions. Later, Hilbert posed his 17th problem, which asked whether for any number of variables, a nonnegative polynomial is a sum of squares of rational functions. This was answered in the affirmative by Artin and Schreier \cite{BCR}. The development of general theory, and the difficulty of Hilbert's proof led to his result on trivariate forms being nearly forgotten.

However, the bounds for the degrees of sums of squares multipliers are very poorly understood in the general case. In fact, the general approach leads to significantly worse bounds than Hilbert's bounds in the case of trivariate forms. We use Theorem \ref{THM Pos Gorenstein}, stated below, to reprove Hilbert's result. We note that in fact Hilbert proved more: the form $pq$ was not just a sum of squares, but a sum of $3$ squares, which allowed Hilbert to conclude that any $p \in P_{3,2d}$ is a sum of squares of at most $3$ rational functions. We do not provide a bound on the number of squares, however, we hope that a simple proof of this result will be useful in understanding the bounds on degrees of sum of squares multipliers.

Our main theorem on the structure of trivariate positive Gorenstein ideals states that such ideals cannot contain positive forms of low degree. This leads to a new, significantly simpler proof of Theorem \ref{THM Hilbert}.

\begin{theorem}\label{THM Pos Gorenstein}
Let $I$ be a positive Gorenstein ideal in $\RR[x_1,x_2,x_3]$ with socle $\ell$ of degree $4d$. Then $I$ does not contain a strictly positive form of degree $2d+2$. 
\end{theorem}


We also consider the question of optimality of the degree of multipliers in Hilbert's theorem. We note that since $P_{3,2}=\Sigma_{3,2}$, Theorem \ref{THM Hilbert} implies that for $p \in P_{3,6}$ there exists $q\in \Sigma_{3,2}$ such that $pq$ is a sum of squares. This is an optimal bound on the degree of $q$ since $P_{3,6}\neq \Sigma_{3,6}$. Similarly, since $P_{3,4}=\Sigma_{3,4}$ we see that for $p \in P_{3,8}$ there exists $q\in \Sigma_{3,4}$ such that $pq$ is a sum of squares. However it is not known whether quadratic multipliers would suffice for degree $8$ ternary forms, and more generally, whether bounds of Theorem \ref{THM Hilbert} are optimal.

Construction (or proof of existence) of nonnegative forms that are not sums of squares is already nontrivial \cite{Rez1}. Our task is to construct forms that are not sums of squares even after multiplication by forms of certain degree. We use an extension of Hilbert's ideas from \cite{Hilbert1} together with tools from convex geometry to prove the following:

\begin{cor}\label{COR Bounds}
For all $k \in \mathbb{N}$ there exist forms $p \in \rx_{3,4k+6}$ such that $pq$ is not a sum of squares for all $q \in \Sigma_{3,2k}$. There exist forms in $P_{4,6}$ and $P_{7,4}$ such that $pq$ is not a sum of squares for all $q \in \Sigma_{4,2}$ (resp. $\Sigma_{7,2}$).
\end{cor}

We note that this implies, in particular, that there exist $p \in \Sigma_{3,10}$ for which quadratic multipliers do not suffice, and $p \in \Sigma_{3,14}$ for which quartic multipliers do not suffice. While Theorem \ref{THM Hilbert} optimally settles the case of $P_{3,6}$ we still leave open the case of $P_{4,4}$, the other smallest case where nonnegative polynomials are not the same as sums of squares. 

\begin{question}
Let $p \in P_{4,4}$. Does there exist $q \in P_{4,2}=\Sigma_{4,2}$ such that $pq$ is a sum of squares?
\end{question}
We note that Corollary \ref{COR Bounds} implies that if we increase the degree from $P_{4,4}$ then quadratic multipliers will not be sufficient.

Theorem \ref{THM Pos Gorenstein} also has interesting consequences in polynomial optimization. It directly leads to the following Corollary:

\begin{cor}\label{COR Square}
Let $p$ be a strictly positive form in $\rx_{3,2d}$. Then $p^2$ lies strictly in the interior of the cone of sums of squares $\Sigma_{3,4d}$.
\end{cor}


We observe that a form $p \in \h$ is in the interior of $\sq$ if and only if there exists a positive definite Gram matrix for $p$ (cf Lemma \ref{LEMMA Gram} and \cite{Gram}). We also note that a feasible semidefinite program will return Gram matrix of maximal rank. Therefore we see that $p \in P_{3,2d}$ is strictly positive if and only if $p^2$ has a Gram matrix of full rank. This provides with a numerical way to test strict positivity of $p \in \rx_{3,2d}$.

This idea can be extended to testing feasibility of systems of equations. Consider a system of equations $f_1(x)=\dots=f_k(x)=0$ with $f_i \in \rx_{3,d}$. The system is infeasible if and only if $q=f_1^2+\dots+f_n^2$ is strictly positive. We can decide strict positivity of $q$ for ternary forms, by checking whether $q^2$ lies on the boundary of $\Sigma_{3,4d}$.

This observation leads to the following question: 

\begin{question}Let $p \in \rx_{n,2d}$ be a strictly positive form. When can $p^2$ lie on the boundary of $\Sigma_{n,4d}$? \end{question}

We show that this is not possible for ternary forms, and results of \cite{GB2012} also show that this is not possible when $p \in \rx_{4,2}$, but the general case is open, and has potentially interesting consequences for deciding complexity of semidefinite programming.

\section{Master Theorem}

We begin by proving our main tool from Cayley-Bacharach duality. Let $\Gamma$ be a finite collection of points in $\CP^{n-1}$, let $I=I(\Gamma)$ be the ideal of $\Gamma$ and let $A(\Gamma)$ be the coordinate ring of $\Gamma$: $A(\Gamma)=\cx/I$. Evaluations on points of $\Gamma$ are linear functionals on $\cx$. Let $\Ind_d(\Gamma)$ be the number of linearly independent evaluation functionals on forms of degree $d$. It follows that $\Ind_d(\Gamma)=\dim A(\Gamma)_d$. Let $\Dep_{d}(\Gamma)$ be the number of linearly dependent relations between the evaluation functionals. Then $\Dep_{d}(\Gamma)=|\Gamma|-\Ind_d{\Gamma}$.
\begin{theorem}\label{THM Master}
Let $p$ and $q$ be two curves in $\CP^2$, with $\deg p=k$, $\deg q=s$ and $k \geq s \geq 3$, intersecting transversely in a $0$-dimensional variety $\Gamma$. Then for any subvariety $\Gamma' \subseteq \Gamma$ such that $\Dep_s(\Gamma') \geq 1$ we have: 
\begin{equation*}\label{EQN master}
\Ind_s(\Gamma')-\Dep_s(\Gamma') \geq \Ind_s(\Gamma)-\Dep_s(\Gamma). \end{equation*}
\end{theorem}
\begin{proof}

We note that by application of Bertini's Theorem we may choose $p$ to be smooth. Let $\Gamma''$ be the complement of $\Gamma'$ in $\Gamma.$ Let $J$, $J'$ and $J''$ be the ideals of $\Gamma$, $\Gamma'$ and $\Gamma''$ respectively. 

By applying Theorem CB7 of \cite{EGH} we see that $\Dep_s(\Gamma')$ is equal to the dimension of the vector space of forms of degree $k-3$ vanishing on $\Gamma''$ modulo vanishing on $\Gamma$. Therefore $\Dep_s(\Gamma')=\dim J''_{k-3}-\dim J_{k-3}.$ 

We know that $\Dep_s(\Gamma') \geq 1$, and therefore there exists a form $r$ of degree $k-3$ vanishing on $\Gamma''$. Let $G$ be the intersection of $r$ with $p$. Since $p$ is a smooth curve, we see that $G$ is a zero dimensional scheme and $G$ contains $\Gamma''$. Let $G'$ be the subscheme residual to $\Gamma''$ in $G$. Let $I'$ be the ideal of curves containing $G'$. We can now apply Theorem CB7 to $\Gamma''$ as a subscheme of $G$. It follows that $\Dep_{k-3}(\Gamma'')$ is equal to the dimension of the vector space of curves of degree $k-3$ containing $G'$ modulo containing all of $G$. Since $r$ is the unique form of degree $k-3$ vanishing on all of $G$ it follows that $\Dep_{k-3} \Gamma''=\dim I'_{k-3}-1$, and also  $J''_{k-3} \cap I'_{k-3}$ is a 1-dimensional vector space spanned by $r$. Therefore we obtain the following inequality:
\begin{equation*}\label{EQN Ind+Dep1}
\dim \cx_{3,k-3} \geq  \dim J''_{k-3}+\Dep_{k-3}(\Gamma'').\end{equation*}
By another application of Theorem CB7 we see that $\dim J'_s-\dim J_s=\Dep_{k-3}(\Gamma'')$ and $\dim J''_{k-3}-\dim J_{k-3}=\Dep_{s}(\Gamma')$.  It follows that 
\begin{equation*}\label{EQN Ind+Dep2}-\Dep_s(\Gamma') \geq   \dim J'_s- \cx_{3,k-3}-\dim J_s+\dim J_{k-3}.\end{equation*}
We now observe that $\Ind_{s}(\Gamma')+\dim J'_s=\dim \cx_{3,s}$ and we have
$$\Ind_s(\Gamma')-\Dep_s(\Gamma') \geq \Ind_s(\Gamma)-\Ind_{k-3}(\Gamma).$$
We now apply Theorem CB7 for the final time to note that $\Ind_{k-3}(\Gamma)=\Dep_s(\Gamma)$ and the Theorem follows.
\end{proof}

\section{Proofs of Dimension Results}

We begin by deriving a codimension bound from Theorem \ref{THM Master}.

\begin{theorem}\label{THM mainrank}
Let $S$ be a subspace of $\cx_{3,d}$ with $d\geq 3$, such that $\V(S)=\emptyset$. Let $I=I(S)$ be the ideal generated by $S$ and suppose that $\codim I_{2d} \geq 1$. Then $$\codim I_d \geq 3d-2.$$ 
The bound is tight when $S$ is the degree $d$ part of the complete intersection of a cubic and two forms of degree $d$.
\end{theorem}

\begin{proof}
Since $\codim I_{2d} \geq 1$ there exists a linear functional $\ell: \cx_{3,2d} \rightarrow \CC$ vanishing on $I_{2d}$. By applying Bertini's theorem we know that there exist two forms $p,q \in S$ such that they intersect transversely in $d^2$ points $\Gamma=\{\gamma_1,\dots,\gamma_{d^2}\} \subset\CP^2$. Let $z_1,\dots,z_{d^2}$ be affine representatives for the points $\gamma_i$. Since the ideal generated by $p$ and $q$ is radical, we know that we can write $\ell$ as a linear combination of points evaluations on the points $z_i$:  $$\ell(f)=\sum_{i=1}^{d^2}\mu_if(z_i)\hspace{5mm} \text{for all} \hspace{5mm} f \in \cx_{3,2d}.$$

Now define a symmetric bilinear $\bl$ on $\cx_{3,d}$ given by $\bl(f,g)=\ell(fg)=\sum_{i=1}^{d^2} \mu_if(z_i)g(z_i).$ It follows that $S$ is contained in the kernel of $\bl$. Let $\Gamma'$ be the subvariety of $\Gamma$ consisting of points $\gamma_i$ for which the coefficient $\mu_i$ is not zero. Evaluation at any non-zero point of $\CC^3$ leads to a rank $1$ symmetric bilinear form. Therefore it follows that $$\rank \bl \geq \Ind_d(\Gamma')-\Dep_{d}(\Gamma').$$

We observe that if $\Dep_d(\Gamma')=0$, then the forms in the kernel of $\bl$ must vanish on all of $\Gamma'$, and therefore forms in $S$ will have a common projective zero. Thus $\Dep_d(\Gamma')\geq 1$ and we can apply Theorem \ref{THM Master}. We see that $\Ind_d{\Gamma'}-\Dep_{d}(\Gamma') \geq \dim \cx_{3,d}-2-(d^2-\dim \cx_{3,d}-2)=3d-2$. Since $S \subseteq \ker \bl$, we see that $\codim S \leq \rank \bl$ and the inequality of the Theorem follows.

The complete (empty) intersection of a form of degree $3$ and two forms of degree $d$ generates a Gorenstein ideal with socle of degree $2d$. Therefore the conditions of the Theorem are satisfied. The degree $d$ part of the ideal has dimension $\binom{d-1}{2}+2$ and thus codimension $3d-2$ in $\cx_{3,d}$.

 \end{proof}

We remark that the idea of representing linear function $\ell$ in terms of point evaluations is well known, and contained for instance in the Apolarity Lemma of \cite{Apolarity}, where necessary and sufficient conditions are provided. We can straightforwardly generalize Theorem \ref{THM mainrank} to more than $3$ variables, while also adding the case of degree $2$ to obtain Theorem \ref{THM mainrank2}.
 


 
\begin{proof}[Proof of Theorem \ref{THM mainrank2}]
We prove the first statement by induction on $n$. The base case $n=3$ is Theorem \ref{THM mainrank}. Let $S$ be a subspace of $\cx_{n+1,d}$ such that $\V(S)=\emptyset$ and $\codim I_{2d}(S) \geq 1$. Suppose that $\codim S < 3d-2$.

Let $W$ be the intersection of $S$ with a copy of $\cx_{n,d}$ obtained by removing $x_{n+1}$. Then $W$ is a subspace of $\cx_{n,d}$ and $\operatorname{codim} W < 3d-2$.

We claim that $\V(W)=\emptyset$. Suppose not, and $p(z)=0$ for some non-zero $z=(z_1,\dots,z_n) \in \mathbb{C}^{n}$ and all $p \in W$. Then $f(z_1,\dots,z_n,0)=0$ for all $f \in S$, which is a contradiction. By the induction assumption we must have $W=\cx_{n,d}$. This argument applies to exclusion of any variable $x_k$. Therefore $S$ contains all monomials of degree $d$ containing at most $n$ variables.  Let $q \in \cx_{n+1,2d}$ be a monomial. We can write $q$ as a product of two monomials of degree $d$, each of which uses only $n$ variables and thus $q \in I_{2d}(S)$. Hence we see that $I_{2d}(S)=\cx_{n+1,2d}$, which is a contradiction.

For the second statement we need to provide the base case $n=4$. Since we have $\V(S)=\emptyset$ we know that $S$ contains a complete intersection of $4$ quadrics. But this complete intersection already generates a Gorenstein ideal with socle of degree $4$. Therefore it follows that $S$ must be equal to the linear span of the four quadrics, and thus $S$ has codimension $6$ in $\cx_{4,4}$. The proof now proceeds by induction in exactly the same way.

\end{proof}

We show the following characterization of maximal positive Gorenstein ideals.

\begin{prop}\label{PROP Maximal}
Let $I$ be a positive Gorenstein ideal with socle $\ell$ of degree $2d$. Then $I$ is maximal if and only if $I_d$ generates $I_{2d}$.
\end{prop}
\begin{proof}
We recall that by Lemma 2.2 of \cite{GB2012}, a positive Gorenstien ideal $I$ with socle of degree $2d$ is maximal if and only if $I_d$ is maximal over all Gorenstein ideals with socle of degree $2d$.

Suppose that $I$ is a positive Gorenstein ideal with socle $\ell \in \h^*$, and $I_{d}$ generates $I_{2d}$. Then $I_d$ is maximal, since $I_d$ uniquely determines $\ell$, and $I_d$ already includes all forms $p \in \hd$ such that $\ell(pq)=0$ for all $q \in \hd$. Therefore $I$ is a maximal positive Gorenstein ideal.

Now suppose that $I$ is a maximal positive Gorenstein ideal, but $I_d$ does not generate $I_{2d}$. Let $J$ be the ideal generated by $I_d$. Then there exist two linearly independent functionals $\ell_1$ and $\ell_2$ vanishing on $J_{2d}$. When we take Gorenstein ideals with socles $\ell_1$ and $\ell_2$ we see that they both include $I_d$ and therefore $I_d$ is not maximal, which is a contradiction.
\end{proof}

We note that the condition that $I_d$ generates $I_{2d}$ does not imply that $I$ is generated in degree $d$; it is possible for $I$ to have additional generators of degree greater than $d$, while $I_d$ generates $I_{2d}$. We are now ready to prove Theorem \ref{THM Main Rank}.



\begin{proof}[Proof of Theorem \ref{THM Main Rank}] 
The fact that the forms in $I_d$ have no common zeroes, real or complex follows from \cite{GB2012}, Corollary 2.3. Now we can apply Theorem \ref{THM mainrank2} to $I_d$ and the dimensional conclusion follows. The statement about maximal positive Gorenstein ideals was proved in Proposition \ref{PROP Maximal}.

To show that the bounds are tight let $\Gamma=\{\gamma_1,\dots,\gamma_{3d}\}$ be a fully real transverse intersection in $\RR\PP^2$ of a cubic $f \in \rx_{3,3}$ and a form $g \in \rx_{3,d}$ of degree $d$. Let $v_1,\dots,v_{3d}$ be affine representatives of the points $\gamma_i$. By Theorem CB7 of \cite{EGH} there exists a unique linear relation for forms in $\hd$ evaluated on the points $v_i$:
$$u_1f(v_1)+\dots+u_{3d}f(v_{3d})=0 \hspace{7mm} \text{for all} \hspace{5mm} f \in \hd,$$
and furthermore all coefficients $u_i \in \RR$ are nonzero. We can now apply exactly the same methods as in Theorem 6.1 of \cite{GB2012}. It follows that any linear functional $\ell \in \rx_{3,2d}^*$ given by $$\ell(f)=\sum_{i=1}^{3d}a_if(v_i) \hspace{7mm} \text{with all} \hspace{2mm} a_i>0 \hspace{2mm} \text{except one}\hspace{3mm}  \text{and} \hspace{2mm} \sum_{i=1}^{3d}\frac{u_i^2}{a_i}=0$$ defines an extreme ray of $\Sigma_{3,2d}^*$, and $\ell \notin P_{3,2d}^*$. We note that $\ell$  is a socle of a maximal positive Gorenstein ideal $I(\ell)$, which is the complete intersection of a cubic with two forms of degree $d$. We can extend $\ell$ to be a linear functional on $\h$ by adding $0$ coordinates to points $v_i$ and the corresponding $\ell$ will still define a maximal positive Gorenstein ideal.

For the case $d=2$, we take a fully real transverse intersection of $3$ quadrics in $\RR\PP^3$ and follow the same construction. This case was explicitly discussed in \cite{GB2012} and the same construction as above will yield an extreme ray of $\Sigma_{4,4}^*$ that is not a point evaluation. The linear functional $\ell$ can again be extended to $\rx_{n,4}$.

\end{proof}
 
Corollary \ref{COR Extreme Rays} immediately follows from Theorem \ref{THM Main Rank}, since $\rank \ql=\codim I_d(\ell)$.

\begin{proof}[Proof of Theorem \ref{THM Moments}]
We know that any $\ell \in \sq^*$ can be written as sum of the extreme rays of $\sq^*$. The kernel of the moment matrix $M_{\ell}$ is equal to $I_d(\ell)$. By Corollary \ref{COR Extreme Rays} the only extreme rays of $\sq^*$ that have moment matrices of rank at most $3d-3$ are point evaluations. Therefore, if $\rank M_{\ell} \leq 3d-3$ then $M_{\ell}$ is a sum of point evaluations. Now suppose that $\ell=\sum \ell_{v_i}$ and $M_{\ell}=\sum M_{v_i}$, where $\ell_{v_i}$ is point evaluation on $v_i \in \RR^n$ and $M_{v_i}$ is its moment matrix. It follows that $\ker M \subset \ker M_{v_i}$ for all $i$. Therefore we can find $\lambda_1$ such that $M_1=M-\lambda_1M_{v_1}$ is positive semidefinite and $\rank M_1=\rank M-1$. Then $M_1$ is again a sum of point evaluations and we can keep reducing the rank of $M_i$ until we obtain a decomposition of $M$ as a sum of exactly $\rank M$ point evaluations.

\end{proof}

\begin{proof}[Proof of Theorem \ref{THM Tensors}] Theorem \ref{THM Tensors} is essentially a restatement of Theorem \ref{THM Moments}. The only claim that is left to be shown is that the bounds are tight: there exist forms $f \in \h$ with positive semidefinite quadratic form $Q_f$ of rank $3d-2$ for $d\geq 3$, $n \geq 3$ and rank $6$ for $d=2$, $n \geq 4$,  such that the real Waring rank of $f$ is strictly greater than $\rank Q_f$.

We take $f \in \h$ such that $\partial f$ is an extreme ray of $\sq^*$ with rank $Q_f=3d-2$ for $d\geq 3$ or rank $6$ for $d=2$, which we know exist by Corollary \ref{COR Extreme Rays}. Since $Q_f$ is positive semidefinite, we know that any decomposition of $\partial f$ as a sum point evaluations must include at least $\rank Q_f$ point evaluations with positive signs.
But $\partial f$ is not a linear combination of point evaluations with positive coefficients, since $\partial f$ is an extreme ray of $\sq^*$ and therefore any decomposition of $f$ as $2d$-th powers of linear forms must include negative signs and we have more than $\rank Q_f$ powers. 
\end{proof}

\section{Positive Forms Contained in Positive Gorenstein Ideals}

We now begin the investigation of positive forms contained in positive Gorenstein ideals. Let $\ell \in \h^*$ be the socle of a positive Gorenstein ideal $I(\ell)$. Suppose that $I(\ell)$ contains a strictly positive form $p$ of degree $2k$. Let $C_p$ be the cone of linear functionals $m$ in $\Sigma_{n,2d}^*$ such that $p$ is contained in $I(m)$:
$$C_p=\{m \in \Sigma_{n,2d}^* \st p \in I(m)\}.$$ 
The condition $p \in I(m)$ is equivalent to $m(pq)=0$ for all $q \in \rx_{n,2d-2k}$.
Therefore the cone $C_p$ is the section of $\Sigma_{n,2d}^*$ with the subspace $L_p$ of $\rx_{n,2d}^*$ consisting of linear functionals vanishing on $p \cdot \rx_{n,2d-2k}$: $$C_p=\Sigma_{n,2d}^* \cap L_p.$$ Thus $C_p$ is a closed convex cone and it contains a nontrivial  linear functional $\ell$. We consider extreme rays of $C_p$.

Recall that to a linear functional $\ell \in \rx_{n,2d}^*$  we can associate a quadratic form $Q_{\ell}$ on $\rx_{n,2d}$ given by $$Q_{\ell}(g)=\ell(g^2) \hspace{3mm} \text{for} \hspace{3mm} g \in \rx_{n,d}.$$
It follows that a linear functional $\ell$ is in $\Sigma_{n,2d}^*$ if and only if the form $Q_{\ell}$ is positive semidefinite.

\begin{lemma}\label{Lemma extreme rays}
Let $\ell$ be an extreme ray of $C_p$ and let $W \subset \rx_{n,d}$ be the kernel of $Q_{\ell}$. Then $p$ and $W$ have no common zeroes (real or complex), i.e. $\V_{\CC}(p,W)=\emptyset$.
\end{lemma}
\begin{proof}
Let $\ell$ and $W \subset \rx_{n,d}$ be as above. Since $p$ is strictly positive it follows that $p$ and forms in $W$ have no common real zeroes. Suppose that there is a common complex zero $z$. We recall that $C_p$ is a section of $\Sigma_{n,2d}^*$ with the subspace $L_p$. It follows that the kernel $W$ of $Q_{\ell}$ is strictly maximal among all the linear functionals in $L_p$, (\cite{GB2012}, Lemma 2.2). In other words, if $s \in L_p$ and $W \subseteq \ker Q_s$ then $\ell=\lambda s$ for some $\lambda \in \RR$.

Now let $s \in \rx_{n,2d}^*$ be given by $s(f)=\operatorname{Re} f(z)$, the real part of $f(z)$. Then $s \in L_p$ and $W \subseteq \ker Q_s$. Therefore $\ell(f)=\lambda \operatorname{Re} f(z)$ for some $\lambda \in \RR$. Thus $Q_{\ell}(g)=\lambda \operatorname{Re} g^2(z)$ for $g \in \rx_{n,d}$. It follows that $Q_{\ell}$ is not positive semidefinite, and we arrive at a contradiction.
\end{proof}

The following is a generalization of Lemma 2.8 of \cite{GB2012}.

\begin{lemma}\label{LEMMA Bertini}
Suppose that $p\in \rx_{n,k}$ and $q_1,\ldots,q_{n-1} \in \hd$ together form a sequence of parameters and $n \geq 3$. Then there exist $f_1, \ldots,f_{n-2}$ in the real linear span of $q_i$ such that the forms $p, f_1, \ldots, f_{n-2}$ intersect transversely in $kd^{n-2}$ (possibly complex) points.
\end{lemma}

\begin{proof}
Let $W$ be the linear span of $q_1, \ldots, q_{n-1}$ with complex coefficients. We begin by showing that there exist linear combinations $f_1,\ldots, f_{n-2} \in W$ such that $p, f_1, \ldots f_{n-2}$ intersect transversely in $\mathbb{CP}^{n-1}$.

Let $\V_0$ be the projective variety defined by $p$. Then $W$ defines a linear system of divisors on $\V_0$. By Bertini's theorem a general element of $W$ intersects $\V_0$ in a smooth variety of dimension $n-3$. Let $f_1$ be such a form in $W$, let $\V_1$ be the smooth variety defined by $p$ and $f_1$, and let $W_1$ be a subspace of $W$ complementary to $f_1$. Then $W_1$ defines a linear system of divisors on $\V_1$ and by Bertini's Theorem the intersection of $\V_1$ with a general element of $W_1$ is a smooth variety of dimension $n-4$. Let $f_2$ be such an element of $W_1$. Now we can let $\V_2$ be the smooth variety defined by $p$, $f_1$ and $f_2$, let $W_2$ be the complementary subspace to $f_1$ and $f_2$ and repeatedly apply Bertini's Theorem until we get a $0$-dimensional smooth intersection. Hence the forms $p, f_1,\dots,f_{n-2}$ we constructed intersect transversely.

Now we argue that there exist \textit{real} linear combinations $f_1,\ldots, f_{n-2}$ which intersect transversely with $p$. Suppose not and let $f_i=\sum_{j=1}^{n}\alpha_{ij}q_j$. Then for all $\alpha_{ij} \in \mathbb{R}$ the forms $p$ and $f_i$ do not intersect transversely. This is an algebraic conditions on the coefficients $\alpha_{ij}$, given by vanishing of some polynomials in the variables $\alpha_{ij}$. However, if a polynomial vanishes on all real points then it must be identically zero. Therefore, no complex linear combinations of $q_i$ intersect transversely, which is a contradiction.
\end{proof}

We recall the following Lemma from \cite{GB2012}, which allows us to limit the number of complex point evaluations that define the socle of a positive Gorenstein ideal.
Let $S$ be a finite set of points in $\co^n$ that is invariant under conjugation: $\bar{S}=S$. Let $S$ be given by $S=\{r_1,\ldots,r_k, z_1,\ldots,z_m, \bar{z}_1,\ldots,\bar{z}_m\}$ with $r_i \in \RR^n$ and $z_i,\bar{z}_i \in \co^n$, $z_i\neq \bar{z}_i$. Let $\ell: \h \rightarrow \RR$ be a linear functional given as a combination of evaluations on $S$:
$$\ell(p)=\sum_{i=1}^k \lambda_ip(r_i)+\sum_{i=1}^m \left(\mu_ip(z_i)+\bar{\mu}_ip(\bar{z}_i)\right),\quad p \in \h$$

\noindent with $\lambda_i \in \RR,$ $\mu_i \in \co$ and $\lambda_i,\mu_i \neq 0$.


\begin{lemma}\label{LEMMA conditions}
Suppose that $\ql$ is positive semidefinite, then the number of complex conjugate pairs in $S$ is at most equal to $\Dep_d(S)$.
\end{lemma}

\section{Proofs of Positive Forms Theorems}


With the preparatory work in the previous section we are in position to prove Theorem \ref{THM Pos Gorenstein}.
\begin{proof}[Proof of Theorem \ref{THM Pos Gorenstein}.]
The case $d=1$ is straightforward since we $4d=2d+2=4$ and $\Sigma_{3,4}=P_{3,4}$ and therefore any linear functional $\ell \in \Sigma_{3,4}^*$ cannot vanish on a strictly positive form of degree $4$. We now consider $d \geq 2$.

Suppose not and let $p$ be a strictly positive polynomial of degree $2d+2$ contained in $I$. Then by Lemma \ref{Lemma extreme rays} we may assume that $p$ and $I_{d}$ have no common zeroes. Therefore, by Lemma \ref{LEMMA Bertini} we can find a form $q \in I_d$ such that  $p$ and $q$ intersect transversely in $(2d+2)2d$ points. Since $p$ is strictly positive on $\rn$ we know that the intersection of $p$ and $q$ consists entirely of complex points.

Let $k=d(2d+2)$ and let $\Gamma=\{\gamma_1,\dots,\gamma_{2k}\} \subset \CP^2$ be the intersection of $p$ and $q$. Let $z_i$ be affine representatives of $\gamma_i$. Since points $\gamma_i$ are strictly complex we can choose $z_i$ in conjugate pairs. 
As before we can express the linear functional $\ell$ as a linear combination of point evaluations on $z_i$: $$\ell(f)=\sum_{i=1}^k \left( \mu_if(z_i)+\bar{\mu}_if(\bar{z}_i)\right) \hspace{5mm} \text{for all} \hspace{5mm} f \in \h.$$
The coefficients $\mu_i$ can be chosen in conjugate pairs since $\ell$ is a real functional. Let $\Gamma'$ be the subset of $\Gamma$ corresponding to nonzero coefficients $\mu_i$. We know by Lemma \ref{LEMMA conditions} that the $\Dep_d(\Gamma')$ is at least $\frac{1}{2}|\Gamma'|$. It follows that $\Dep_d(\Gamma') \geq 1$ and we can apply Theorem \ref{THM Master} to $\Gamma'$. We see that $\Ind_d(\Gamma')-\Dep_d(\Gamma') \geq \binom{2d+2}{2}-\binom{2d+1}{2}-1=2d$. Since $\Ind_d(\Gamma')+\Dep_d(\Gamma')=|\Gamma'|$ and $\Dep_d(\Gamma') \geq \frac{1}{2}|\Gamma'|$ we reach a contradiction. 

\end{proof}


We can quickly deduce Theorem \ref{THM Hilbert} from Theorem \ref{THM Pos Gorenstein}.

\begin{proof}[Proof of Theorem \ref{THM Hilbert}]
Suppose not and for some $p \in P_{3,2d}$ we have $pq \notin \Sigma_{3,4d-4}$ for all $q \in P_{3,2d-4}$. Let $T$ be the set of nonnegative forms $f$ of degree $2d$ such that $fg$ is a sum of squares for some nonnegative form $g$ of degree $2d-4$:
$$T=\{f \in P_{3,2d} \st fg \in \Sigma_{3,4d-4} \hspace{3mm} \text{for some} \hspace{3mm} g \in P_{3,2d-4}\}.$$ It is easy to see that the set $T$ is closed. Since $T \subset P_{3,2d}$ there must exist a strictly positive form $p \notin T$.

We have $pq \notin \Sigma_{3,4d-4}$ for all $q \in P_{3,2d-4}$ and therefore $pq \notin \Sigma_{3,4d-4}$ for all non-zero $q \in \rx_{3,2d-4}$. In other words, the linear subspace $p\cdot \rx_{3,2d-4}$ of $\rx_{3,4d-4}$ intersects the convex cone $\Sigma_{3,4d-4}$ only at the origin. It follows that there exists a linear functional $\ell \in \Sigma_{3,4d-4}^*$ such that $\ell$ is zero on $p \cdot \rx_{3,2d-4}$. In other words there exists a positive Gorenstein ideal with socle $\ell$ of degree $4d-4$ which contains a strictly positive form $p$ of degree $2d$. This is a contradiction by Theorem \ref{THM Pos Gorenstein}.
\end{proof}

\subsection{Positive Forms on the Boundary of $\Sigma_{n,2d}$}

Corollary \ref{COR Square} follows immediately from Theorem \ref{THM Pos Gorenstein}: if $p$ is a strictly positive form on the boundary of $\Sigma_{3,4d}$, then $p$ be cannot be a square of a form in $\rx_{3,d}$. In fact we can do better. 

\begin{cor}
Let $p$ be a strictly positive form on the boundary of $\Sigma_{3,2d}$. Then $p$ is a sum of at most $\binom{d+2}{2}-3d+2$ squares and $p$ cannot be written as a sum of fewer than $3$ squares.
\end{cor}
\begin{proof}
Let $p$ be a strictly positive form on the boundary of $\Sigma_{3,2d}$ with $d \geq 3$. Then there exists an extreme ray $\ell$ of the dual cone $\Sigma_{3,2d}^*$, such that $\ell(p)=0$, defining a maximal positive Gorenstein ideal $I(\ell)$. Now suppose that $p=\sum f_i^2$ for some $f_i \in \rx_{3,d}$. It follows that $\ql(f_i)=0$ for all $i$, and since $\ql$ is a positive semidefinite quadratic form we see that all $f_i$ lie in the kernel $I_{d}(\ell)$ of $\ql$. By Theorem \ref{THM Main Rank} we know that $\dim I_d(\ell) \leq \binom{d+2}{2}-3d+2$ and the upper bound follows.


Now suppose that $p$ is a sum of $2$ squares, $p=f_1^2+f_{2}^2$. Since $p$ is strictly positive we know that the forms $f_i$ have no common real zeroes. Therefore we found two forms $f_i \in I_d(\ell)$ that have no common real zeroes. This is not possible when $d$ is odd. For even $d$ from the proof of Lemma \ref{LEMMA Bertini} we know that $2$ generic forms in $I_d(\ell)$ intersect transversely, and hence we can find forms $f_i' \in \wl$ in a neighborhood of $f_i$  such that $f_i'$ intersect transversely in $d^2$ complex points. This is a contradiction by Theorem \ref{THM Master} and Lemma \ref{LEMMA conditions}.
\end{proof}

We observe that for $d=3$ the upper bound is equal to the lower bound and therefore the bounds are tight. This was discussed in \cite{GB2012} and used in \cite{HCases} to study the algebraic boundary of $\Sigma_{3,6}$. The lower bound is always tight, i.e. we can always find a positive form on the boundary of $\Sigma_{3,2d}$ which is a sum of 3 squares. In fact this happens for any maximal positive Gorenstein ideal $I$ whose socle is not a point evaluation. We know that the forms in $I_d$ have no common zeroes by Theorem \ref{THM Main Rank}. Therefore we can find three forms $f_1,f_2,f_3 \in I_d$ with no common zeroes and $f_1^2+f_2^2+f_3^2$ will be strictly positive and on the boundary of $\Sigma_{3,2d}$.

We remark that the upper bound of $\binom{d+2}{2}-3d+2$ is also sharp for $d=4$, where it is equal to $5$. This can be seen by considering the maximal positive Gorenstein ideals constructed in the proof of Theorem \ref{THM Main Rank}, which come from a complete intersection of a cubic and two quartics. However we expect that the bound is not optimal for higher $d$.

\noindent \textbf{Connection to Semidefinite Programming.} Let $s=\dim \hd$, let $\mathcal{B}=\{f_1,\dots,f_s\}$ be a basis of $\h$ and let $b=(f_1,\dots,f_s)$. For $p \in \rx$ let $c(p)=\{c_1,\dots,c_s\}$ be the vector of coefficients expressing $p$ as a linear combination of $f_i$: $p=\sum_{i=1}^sc_if_i$. The following is a well known connection between sums of squares and positive semidefinite matrices:
\begin{lemma}\label{LEMMA Gram}
Let $p \in \h$. The form $p$ is a sum of squares if and only if there exists a positive semidefinite matrix $M$ such that $p=b^TMb$. Furthermore $p$ is in the interior of $\sq$ if and only if there exists a positive definite $M$ with the above property.
\end{lemma}
\begin{proof}
The key observation is that $M$ is rank one psd matrix, then $M=c(p)c(p)^T$ for some vectors $c(p)$ and therefore $b^TMb=p^2$. Thus if $p=\sum q_i^2$ then we can take $M=\sum c(q)c(q)^T$. On the other hand, any psd matrix can be written as a sum of rank 1 psd matrices, and therefore if $M$ is psd then $p=b^TMb$ is a sum of squares. This shows that $\sq$ is a linear projection of the cone of $s \times s$ psd matrices. We now observe the interior of the cone of $s \times s$ psd matrices consists of positive definite matrices. A linear projection an interior point to the boundary of the image and therefore the second implication follows.
\end{proof}
\section{Multiplier Discussion}



The following is a generalization of the ideas of Hilbert's original method of showing that there exist nonnegative forms that are not sums of squares. We show that values on complete intersections can be also used to certify that a nonnegative form has no sum of squares multipliers $q$ of some degree, such that $pq$ is a sum of squares. 



Let $S=\{v_1,\dots,v_m\}$ be a finite subset of $\RR^n$. Let $E_S$ be the evaluation map that sends $p \in \rx$  to its values on the points $v_i$:
$$E_S:\rx \longrightarrow \rmm, \hspace{8mm} E_S(p)=(p(v_1),\dots,p(v_m)).$$


Let $P_{2d}(S)$ and $\Sigma_{2d}(S)$ be the images of $\p$ and $\sq$ under $E_S$ respectively, and let $H_{d}(S)$ be the image of $\hd$. 
We first show that the cone $\Sigma_{2d}(S)$ is always closed.
\begin{lemma}\label{LEMMA SOS is Closed}
$\Sigma_{2d}(S)$ is a closed convex cone for all $d$ and all finite subsets $S$ of $\RR^n$.
\end{lemma}
\begin{proof}
Consider $H_d(S)$, the image of $\hd$ under the evaluation projection $E(S)$. Let $\mathbb{S}_d$ be the unit sphere in $H_d(S)$ with respect to the standard inner product on $\RR^m$. Let $C$ be the points in $\RR^m$ that are pointwise squares of points in $\mathbb{S}_d$:
$$C:=\{(x_1,\dots,x_m) \in \rmm \,\, \mid\,\, (x_1,\dots,x_m)=(v_1^2,\dots,v_m^2) \hspace{3mm} \text{for some} \hspace{3mm} (v_1,\dots,v_m) \in \mathbb{S}_d\}.$$
We see that $C$ is a compact subset of $\rmm$. Let $K$ be the convex hull of $C$. It follows that $K$ is a compact convex set and $0 \notin K$. We now observe that $\Sigma_{2d}(S)$ is the conical hull of $K$. Therefore $\Sigma_{2d}(S)$ is a closed convex cone.
\end{proof}

We now restrict ourselves to finite subsets of $\RR^n$ coming as affine representatives of a complete intersection. Let $\Gamma=\{\gamma_1,\dots,\gamma_m\}$ be a fully real transverse intersection of $n-1$ forms from $\hd$ in $m=d^{n-1}$ points in $\RR\mathbb{P}^{n-1}$. Let $S=\{v_1,\ldots,v_m\} \subset \rn$ be a set of affine representatives for $\gamma_i$. The following theorem is a special case of Theorem 3.1 in \cite{GB2012}.
\begin{theorem}\label{THM Values of Nonnegative Forms}
Let $\RR^m_{++}$ be the positive orthant of $\RR^m$. The intersection of $H_{2d}(S)$ with the positive orthant is contained in $P_{2d}(S)$: $$H_{2d}(S) \cap \RR^m_{++} \subset P_{2d}(S).$$
\end{theorem}

We now show how to use values of forms on a complete intersection to show existence of nonnegative forms that have no sum of squares multipliers of certain degree.  

\begin{theorem}\label{THM Multipliers}
Let $\Gamma$ be a generic completely real transverse intersection of $n-1$ forms from $\hd$ 
in $d^{n-1}$ points in $\RR\mathbb{P}^{n-1}$. 
The genericity assumption on $\Gamma$ is that for all $t \leq 2d$, any collection of at most $\Ind_{t}(\Gamma)$ points of $\Gamma$ imposes linearly independent conditions on forms in $\rx_{n,t}$. 
Suppose that  $$\Ind_{2d}(\Gamma)\geq \Ind_{d+k}(\Gamma)+\Ind_{k}(\Gamma).$$ Then there exists a nonnegative form $p \in P_{n,2d}$ such that $pq$ is not a sum of squares for any sum of squares $q \in \Sigma_{n,2k}$.
\end{theorem}

\begin{proof}


Suppose not and for every nonnegative form $p \in P_{n,2d}$ there exists $q \in P_{n,2k}$ such that $pq$ is a sum of squares. Let $S \subset \RR^n$ be the set of affine representatives of $\Gamma$. 
From Lemma \ref{LEMMA SOS is Closed} we know that $\Sigma_{2k}(S)$ and $\Sigma_{2d+2k}(S)$ are closed convex cones.

For $x,y \in \RR^m$ let $x \cdot y$ denote pointwise multiplication of $x$ and $y$: $x \cdot y=(x_1y_1,\dots,x_my_m)$. Let $\RR^m_+$ be the nonnegative orthant of $\RR^m$. We claim that for any $x \in H_{2d}(S) \cap \RR^m_+$ there exists $y \in \Sigma_{2k}(S)$ such that $x\cdot y \in \Sigma_{2d+2k}(S)$. By our assumption, and Theorem \ref{THM Values of Nonnegative Forms} we know that this holds for all $x \in H_{2d}(S) \cap \RR^m_{++}$.

Let $x \in H_{2d}(S) \cap \RR^m_+$ and let $x_1,x_2, \dots$ be a sequence of points in $H_{2d}(S) \cap \RR^m_{++}$ approaching $x$. Let $y_1,y_2, \dots$ be a sequence of multipliers in $\Sigma_{2k}(S)$ such that $x_i \cdot y_i\in \Sigma_{2d+2k}(S)$. Since the condition of belonging to $\Sigma_{2d+2k}(S)$ is conical, it follows that we may choose $y_i$ lying on the unit sphere in $\RR^m$. Let $y$ be an accumulation point of $y_i$. Since the cone $\Sigma_{2k}(S)$ is closed, it follows that $y \in \Sigma_{2k}(S)$ and since $\Sigma_{2d+2k}(S)$ is closed we see that $x \cdot y \in \Sigma_{2d+2k}(S)$.

We now examine $T=H_{2d}(S) \cap \RR^m_{+}$. $T$ is a polyhedral cone and we have $\dim H_{2d}(S)=\Ind_{2d}(\Gamma)$. Let $F$ be a face of $T$ of codimension $\Ind_{d+k}(\Gamma)$. Then $\Ind_{d+k}(\Gamma)$ linearly independent defining inequalities of $T$ are tight on $F$. Therefore for any $u$ in the relative interior of $F$, the coordinates of $u$ are zero in at least $\Ind_{d+k}(\Gamma)$ entries corresponding to the tight inequalities on $F$, and are non-zero in at least $\Ind_{2d}(\Gamma)-\Ind_{d+k}(\Gamma)$ entries corresponding to the points where evaluation is linearly independent from the zeroes of $u$. By relabeling, if necessary, let $B=\{v_1,\dots,v_{\Ind_{d+k}}\}$ be the collection points corresponding to the tight  linearly independent inequalities of $u$. 

We know from above that there exists $y \in \Sigma_{2k}(S)$ such that $u\cdot y \in \Sigma_{2d+2k}(S)$. By genericity of $\Gamma$ we know that any $f\in \rx_{n,d+k}$ that vanishes on $B$ must vanish on all of $\Gamma$. Therefore we see that $u \cdot y=0$. Therefore $y$ comes from a form $r \in \Sigma_{n,2k}$ that vanishes on the $\Ind_{2d}(\Gamma)-\Ind_{d+k}(\Gamma)$ points in $\Gamma$. Let $r=\sum r_i^2$. Then each $r_i \in \rx_{n,k}$ must vanish on $\Gamma$. Since we have $k<d$, it follows that all $r_i$ are identically zero, which is a contradiction.


\end{proof}

We now use Theorem \ref{THM Multipliers} to establish Corollary \ref{COR Bounds}.

\begin{proof}[Proof of Corollary \ref{COR Bounds}] First let $\Gamma$ be a fully real generic intersection of two curves from $\rx_{3,2k+3}$. Then $\Ind_{k}(\Gamma)=\dim \rx_{3,k}=\binom{k+2}{2}$ and $\Ind_{3k+3}(\Gamma)=\dim \rx_{3,3k+3}-2\dim \rx_{3,k}=\binom{3k+5}{2}-2\binom{k+2}{2}$. Thus $\Ind_{k}(\Gamma)+\Ind_{3k+3}(\Gamma)=4k^2+12k+9=(2k+3)^2$. Also, $\Ind_{4k+6}(\Gamma)=\dim \rx_{3,4k+6}-2\dim \rx_{3,2k+3}+1=\binom{4k+8}{2}-2\binom{2k+5}{2}+1=(2k+3)^2$. Therefore we can apply Theorem \ref{THM Multipliers} with $d=2k+3$.

Next let $\Gamma$ be a fully real generic intersection of three cubics from $\rx_{4,3}$. Then $\Ind_1(\Gamma)=4$ and $\Ind_4(\Gamma)=\dim \rx_{4,4}-3\cdot4=23$. Also $\Ind_{6}(\Gamma)=\dim \rx_{4,6}-3\dim \rx_{4,3}+\binom{3}{2}=27$. Thus we can apply Theorem \ref{THM Multipliers} with $d=3$ and $k=1$.

Finally let $\Gamma$ be a fully real generic intersection of six quadrics from $\rx_{7,2}$. Then $\Ind_{1}(\Gamma)=7$ and $\Ind_3(\Gamma)=\dim \rx_{7,3}-6\cdot7=42$. Also $\Ind_{4}(\Gamma)=\dim \rx_{7,4}-6\dim \rx_{7,2}+\binom{6}{2}=57$. Thus we can apply Theorem \ref{THM Multipliers} with $d=2$ and $k=1$.

\end{proof}

\end{document}